\documentclass[10pt,letterpaper,dvips]{article}

 \usepackage{amssymb,latexsym,amsmath,amsthm}

\usepackage{fullpage} 
\newtheorem{thm}{Theorem}
\newtheorem{prop}{Proposition}
\newtheorem{dfn}{Definition}

\newtheorem{lemma}{Lemma}

\newtheorem{cor}{Corollary}

\numberwithin{equation}{section}

\begin{document}
\date{}

\title{Uniqueness of solutions for a nonlocal elliptic eigenvalue problem}

\author{Craig Cowan\\
{\it\small Department of Mathematical Sciences}\\
{\it\small University of Alabama in Huntsville}\\
{\it\small 258A Shelby Center}\\
\it\small Huntsville, AL 35899 \\
{\it\small ctcowan@stanford.edu}\vspace{1mm}\and
Mostafa Fazly\thanks{This work is supported by a University Graduate Fellowship and is part of the second author's Ph.D. dissertation in preparation under the supervision of Professor N. Ghoussoub.}\hspace{2mm}\\
{\it\small Department of Mathematics}\\
{\it\small University of British Columbia,}\\
{\it\small Vancouver BC Canada V6T 1Z2}\\
{\it\small  fazly@math.ubc.ca}\\
}

\maketitle

\def\L{ {\mathcal{L}}}
\def\d{ \partial }
\def\Na{{\mathbb{N}}}

\def\Z{{\mathbb{Z}}} 

\def\IR{{\mathbb{R}}}

\def\CC{{ C_c^\infty(\Omega)}}

\def\XX{{  \left\{ \psi \in C^2(\overline{\Omega}) :  \psi =0 \quad \pOm \right\} }}

\def\pOm{ \partial \Omega}

\def\HA{ (-\Delta)^\frac{1}{2}}

\def\HAN{ (-\Delta)^{-\frac{1}{2}}}

\newcommand{\E}[0]{ \varepsilon}

\newcommand{\la}[0]{ \lambda}

\newcommand{\s}[0]{ \mathcal{S}}

\newcommand{\AO}[1]{\| #1 \| }

\newcommand{\BO}[2]{ \left( #1 , #2 \right) }

\newcommand{\CO}[2]{ \left\langle #1 , #2 \right\rangle}

\newcommand{\co}[1]{ #1^{\prime}}

\newcommand{\p}[0]{ p^{\prime}}

\newcommand{\m}[1]{   \mathcal{ #1 }}

\begin{abstract}  We examine equations of the form 
\begin{eqnarray*}
\left\{ \begin{array}{lcl}
\hfill  \HA u &=& \lambda g(x) f(u) \qquad \text{in}\  \Omega \\
\hfill u&=& 0 \qquad \qquad \qquad  \text{on}\   \pOm,
\end{array}\right.
  \end{eqnarray*} where $ \lambda >0$ is a parameter and $ \Omega$ is a smooth  bounded domain in $ \IR^N$, $ N \ge 2$.  Here $ g$ is a positive function and $ f$ is an increasing, convex function with $ f(0)=1$  and either $ f$ blows up at $1$ or $ f$ is superlinear at infinity.   We show that the extremal solution $u^*$ associated with the extremal parameter $ \lambda^*$ is the unique solution.  We also show that when $f$ is suitably supercritical and $ \Omega$ satisfies certain geometrical conditions then there is a unique solution for small positive $ \lambda$.

\end{abstract}

\section{Introduction}

We are interested in the following  nonlocal eigenvalue problem 
\begin{eqnarray*}
\hbox{$(P)_\lambda$}\hskip 50pt \left\{ \begin{array}{lcl}
\hfill  \HA u &=& \lambda g(x) f(u) \qquad \text{in}\  \Omega \\
\hfill u&=& 0 \qquad \qquad \qquad  \text{on}\   \pOm ,
\end{array}\right.
  \end{eqnarray*}
 where $ (-\Delta)^\frac{1}{2}$ is the square root of the Laplacian operator,   $ \lambda > 0$ is a parameter, $ \Omega$ is a smooth  bounded domain in  $ \IR^N$ where $ N \ge 2$,   and where $ 0 < g(x) \in C^{1,\alpha}(\overline{\Omega})$ for some $ 0 < \alpha$.     The nonlinearity  $f$ satisfies  one of the following two conditions:

(R) \quad $f$ is smooth, increasing and convex on  $\mathbb{R}$ with $ f(0)=1$ and $ f$ is superlinear at $ \infty$ (i.e. $ \displaystyle \lim_{t \rightarrow \infty} \frac{f(t)}{t}=\infty$),  or 

(S) \quad $f$ is smooth, increasing, convex on $[0, 1)$ with $ f(0)=1$ and $\displaystyle  \lim_{t \nearrow 1} f(t)=+\infty$.

In this paper we  prove there is a unique solution of $(P)_\lambda$ for two parameter ranges:  for small $ \lambda$ and for $ \lambda=\lambda^*$ where $ \lambda^*$ is the so called extremal parameter associated with $(P)_\lambda$. First, let us to recall various known facts concerning the second order analog  of $ (P)_\lambda$.   \\

\textbf{Some notations:}  Let $F(t):=\int_0^t f(\tau) d \tau$ and $ C_f:=\int_0^{a_f} f(t)^{-1} dt$ where $ a_f= \infty$ (resp. $ a_f=1$) when $ f$ satisfies (R)  (resp. $f$ satisfies (S)).  We say a positive function $f$ defined on an interval $I$ is logarithmically convex (or log convex) provided $ u \mapsto \log( f(u) )$ is convex on $I$.   Also, $ \Omega$ will always denote a smooth bounded domain in $ \IR^N$ where $N \ge 2$.

  \subsection{The local eigenvalue problem} 
For a nonlinearity $ f$ which satisfies (R) or (S), the following  second order analog of $(P)_\lambda$ with the Dirichlet boundary conditions 
  \begin{eqnarray*}
\hbox{$(Q)_{\lambda}$}\hskip 50pt \left\{ \begin{array}{lcl}
\hfill   -\Delta u  &=& \lambda f(u) \qquad \text{in}\ \Omega \\
\hfill u&=& 0 \qquad\qquad \text{on}\ \pOm ,
\end{array}\right.
  \end{eqnarray*}
  is by now quite well understood whenever $ \Omega$ is a bounded smooth domain in $ \mathbb{R}^N$. See, for instance, \cite{bcmr,BV,Cabre,CC,egg,mems_book,gg,martel,MP,Nedev}. We now list the  properties one comes to expect when studying $(Q)_{\lambda}$.

   It is well known that there exists a critical parameter  $ \lambda^* \in (0,\infty)$ such that for all $ 0<\lambda < \lambda^*$ there exists a smooth, minimal solution $u_\lambda$ of $(Q)_\lambda$.  Here the  minimal solution means in the pointwise sense.  In addition for each $ x \in \Omega$ the map $ \lambda \mapsto u_\lambda(x)$ is increasing in $ (0,\lambda^*)$.   This allows one to define the pointwise limit $ u^*(x):= \lim_{\lambda \nearrow \lambda^*} u_\lambda(x)$  which can be shown to be a weak solution, in a suitably defined sense, of $(Q)_{\lambda^*}$.
   It is also known that for $ \lambda >\lambda^*$ there are no weak solutions of $(Q)_\lambda$.  Also, one can show that the minimal solution $ u_\lambda$  is a semi-stable solution of $(Q)_\lambda$  in the sense that
   \[ \int_\Omega \lambda f'(u_\lambda) \psi^2 \le \int_\Omega | \nabla \psi|^2, \qquad \forall \psi \in H_0^1(\Omega).\]

   We now come to the results known for $(Q)_\lambda$ which we are interested in extending  to $(P)_\lambda$.
  In \cite{martel} it was shown that the extremal solution $u^*$ is the unique weak solution of $(Q)_{\lambda^*}$.  Some of the techniques involve using concave cut offs which do not seem to carry over to the nonlocal setting.    Here we use some techniques developed in \cite{berc} which were used in studying a fourth order analog of $(Q)_\lambda$.   In  \cite{ddgm} the uniqueness of the extremal solution for $ \Delta^2 u = \lambda e^u$ on radial domains with Dirichlet boundary conditions was shown and this was extended to log convex (see below) nonlinearities in \cite{luo}.  Some of the methods used in \cite{luo} were inspired by the techniques of \cite{berc} and so will ours in the case where $ f$ satisfies (R).  In \cite{cog} it was shown that the extremal solution associated with $ \Delta^2 u = \lambda (1-u)^{-2}$ on radial domains is unique and our methods for nonlinearities satisfying (S) use some of their techniques.   
  
 In \cite{jmq} and \cite{schmitt}  a  generalization of $(Q)_\lambda$ was examined.  They
showed that if $f$ is suitably supercritical at infinity and if $ \Omega$ is a star-shaped domain, then for small $ \lambda>0$ the minimal solution is the unique solution of $(Q)_\lambda$.  In \cite{eg} this was done for a particular nonlinearity $f$  which satisfies (S).   One can weaken the   star-shaped assumption and still have uniqueness, see  \cite{shaaf},  but we do not pursue this approach here.   In section 3 we extend these results to $(P)_\lambda$.     For more results  the on uniqueness of solutions for various elliptic problems involving parameters, see \cite{ds}.

For questions on the regularity of the extremal solution in fourth order problems we direct the interested reader to  \cite{ceg}. 
We also mention the recent preprint \cite{craig_10} which examines the same issues as this paper but for equations of the form $ \Delta^2 u = \lambda f(u)$  in $ \Omega$ with either the Dirichlet boundary conditions $ u=| \nabla u|=0$  on $ \pOm$ or the Navier boundary conditions $ u=\Delta u =0$ on $ \pOm$.    Elliptic systems of the form $ -\Delta u = \lambda f(v)$, $ -\Delta v = \gamma g(u)$ in $ \Omega$ with $ u=v=0$ on $ \pOm$ are also examined.

  \subsection{The nonlocal eigenvalue problem} \label{non} 
We first give the needed background regarding $\HA$ to examine $(P)_\lambda$, for a more detailed background see \cite{ct}.  In \cite{cdds} they examined the problem $(P)_\lambda$ with $ (-\Delta)^s$ replacing $ \HA$ and with $ g(x)=1$.  They did not investigate the questions we are interested in but they did develop much of the needed theory to examine $(P)_\lambda$ and so we will use many of their results.

  There are various ways to make sense of $ \HA u$.     Suppose that $ u(x)$ is a smooth function defined in $ \Omega$ which is zero on $ \pOm$   and suppose that $ u(x) = \sum_k a_k \phi_k(x)$  where $ (\phi_k, \lambda_k)$ are the eigenpairs of $ -\Delta$ in $H_0^1(\Omega)$ which are $L^2$ normalized.    Then  one defines
  \[ \HA u(x) =  \sum_k a_k \sqrt{ \lambda_k} \phi_k(x).\]    Another way is to suppose we are given $ u(x)$ which is zero on $ \pOm$ and  we let $ u_e=u_e(x,y)$ denote a solution of
  \begin{eqnarray*}
 \left\{ \begin{array}{lcl}
\hfill  \Delta u_e &=& 0 \qquad \qquad \text{in}\ \ \mathcal{C}: =\Omega \times (0,\infty)  \\
\hfill  u_e&=& 0 \qquad \qquad \text{on}\ \ \partial_L \mathcal{C}:= \pOm \times (0,\infty) \\
 \hfill   u_e&=& u(x) \qquad \  \text{in}\ \ \Omega \times \{ 0\}.
\end{array}\right.
  \end{eqnarray*}
    Then we define
  \[ \HA u(x) = \partial_\nu u_e(x,y) \big|_{y=0},\]  where $ \nu$ is the outward pointing normal on the bottom of the cylinder, $ \mathcal{C}$.          We call $ u_e$ the harmonic extension of $u$.   We define  $H^1_{0,L}(\mathcal{C})$ to be the completion of $ C_c^\infty( \Omega \times [0,\infty))$ under the norm $ \| u\|^2:= \int_{\mathcal{C}} | \nabla u|^2$.
When working on the cylinder generally we will write integrals of the form $ \int_{\Omega \times \{y=0\}} \gamma(u_e)$ as $ \int_\Omega \gamma(u)$.

  Some of our results require one to examine quite weak notions of solutions to $(P)_\lambda$ and so we begin with our definition of a weak solution.

 \begin{dfn}  Given $ h(x) \in L^1(\Omega)$ we say that $ u \in L^1(\Omega) $ is a weak solution of
 \begin{eqnarray} \label{linear}
 \left\{ \begin{array}{lcl}
\hfill  \HA u &=& h(x) \qquad \text{in}\ \Omega \nonumber  \\
\hfill   u&=& 0 \qquad \quad \text{on}\ \pOm,
\end{array}\right.
  \end{eqnarray}
 provided that
 \[  \int_\Omega u  \psi = \int_\Omega h(x) \HAN \psi \qquad \forall  \psi \in C_c^\infty(\Omega).\]

 \end{dfn}  
 \noindent Here $ \HAN \psi$ is given by the function $ \phi$ where
   \begin{eqnarray*}
 \left\{ \begin{array}{lcl}
\hfill   \HA \phi &=& \psi \qquad \text{in}\ \Omega \\
\hfill  \phi&=& 0 \qquad \text{on}\ \pOm.
\end{array}\right.
  \end{eqnarray*}
 
   The following is a weakened  special case of a  lemma taken from \cite{cdds}.  
\begin{lemma}\label{max}
 Suppose that $ h \in L^1(\Omega)$.    Then there exists a unique weak solution $u$ of (\ref{linear}).  Moreover if $ 0 \le h$ a.e. then $ u \ge 0$ in $ \Omega$.
\end{lemma}

 \begin{dfn} Let $f$ be a nonlinearity satisfying (R).
 \begin{itemize} 
  \item  We say that $ u(x) \in L^1(\Omega)$ is a weak solution of $(P)_\lambda$ provided $ g(x) f(u) \in L^1(\Omega) $, and
 \[ \int_\Omega u \psi = \lambda \int_\Omega  g(x) f(u) \HAN \psi \qquad \forall  \psi \in C_c^\infty(\Omega).\]

 \item We say $ u $ is a regular energy solution of $(P)_\lambda$ provided that $u$ is bounded,  the harmonic extension $u_e$ of $ u$,  is an element of  $ H_{0,L}^1(\mathcal{C})$ and satisfies
 \begin{equation} \label{energy}
 \int_{\mathcal{C}} \nabla u_e \cdot \nabla \phi = \lambda \int_{\Omega } g(x) f(u) \phi,
 \end{equation}  for all $ \phi \in H_{0,L}^1(\mathcal{C})$.

 \item  We say $ \overline{u}$ is a regular energy supersolution of $(P)_\lambda$ provided  that $0 \le \bar{u}$ is bounded,  the harmonic extension of $\bar{u}$ is an element of  $ H_{0,L}^1(\mathcal{C})$ and satisfies  \begin{equation} \label{energy_super}
 \int_{\mathcal{C}} \nabla \overline{u}_e \cdot \nabla \phi \ge  \lambda \int_{\Omega } g(x) f(\overline{u}) \phi,
 \end{equation}  for all $ 0 \le \phi \in H_{0,L}^1(\mathcal{C})$.

 \end{itemize}

 \end{dfn}   
 
 In the case where $f$ satisfies (S) a few minor changes are needed in the definition of solutions.  For a weak solution $u$ one requires that $ u \le 1$ a.e. in $ \Omega$.  For $u$ to be a regular energy solution one requires that $ \sup_\Omega u <1$. 

 We will need the following monotone iteration result, see \cite{cdds}.     Suppose that $ \underline{u}$ and $ \overline{u}$ are regular energy sub and supersolutions of $ (P)_\lambda$.   Then there exists a regular energy solution $u$ of $ (P)_\lambda$   and $ \underline{u} \le u \le \overline{u}$ in $ \Omega$.    By a regular energy subsolution we are using the natural analog  of regular energy supersolution.

We define  the extremal parameter
\[ \lambda^*:= \sup \left\{ 0 \le \lambda:  \mbox{ $ (P)_\lambda$ has a regular energy solution} \right\},\] and we now show some basic properties.

\begin{lemma}\label{basic}

\begin{enumerate}

 \item[(1)] Then $ 0 < \lambda^*$.

\item[(2)]    Then $ \lambda^* < \infty$.

 \item[(3)] For $ 0 < \lambda < \lambda^*$ there exists a regular energy solution $ u_\lambda$ of $(P)_\lambda$ which is minimal and semi-stable.

\item[(4)] For each $ x \in \Omega$ the map  $ \lambda \mapsto u_\lambda(x)$ is increasing on $ (0,\lambda^*)$ and hence  the pointwise limit $ u^*(x):=\lim_{\lambda \nearrow \lambda^*} u_\lambda(x)$ is well defined.   Then $u^*$ is a  weak solution of $ (P)_{\lambda^*}$ and satisfies  $ \int_\Omega g(x) f'(u^*) f(u^*) dx < \infty$.

\end{enumerate}

\end{lemma} 

In this paper we do not need the notion of a semi-stable solution other than for the proof of (4).  For the definition of a semi-stable solution one can either use a nonlocal notion, see \cite{cdds} or instead work on the cylinder which is what we choose to do.  We say that a regular energy solution $ u$ of $(P)_\lambda$ is semi-stable provided that 
\begin{equation} \label{loc_stable}
\int_{\mathcal{C}} | \nabla \phi|^2 \ge \lambda \int_{\Omega} g(x) f(u) \phi^2 \qquad \forall \phi \in H^1_{0,L}(\mathcal{C}).
\end{equation} We now prove the lemma. \\
\textbf{Proof:} (1)  Let $ \overline{u}$ denote a solution of  $ \HA \overline{u}= t g(x)$ with $ \overline{u}=0$ on $ \pOm$ where $ t>0$ is small enough such that $ \sup_\Omega \overline{u}<1$.       One sees that $ \overline{u}$ is a regular energy supersolution of $ (P)_\lambda$ provided $ t \ge \lambda \sup_\Omega f(\overline{u})$ which clearly holds for small positive $ \lambda$.    Zero is clearly a regular energy subsolution and so we can apply the monotone iteration procedure to obtain a regular energy solution and hence $ \lambda^* >0$.

(2) Suppose that either $ f$ satisfies (R) and $ C_f <\infty$ or $ f$ satisfies (S) and so trivially $ C_f < \infty$.  

 Let $ u $ denote a regular energy solution of $ (P)_\lambda$ and let $ u_e$ denote the harmonic extension.    Let $ \phi$ denote the first eigenfunction of $ -\Delta $ in $H_0^1(\Omega)$ and let $ \phi_e$ be its harmonic extension;  so $ \phi_e(x,y)= \phi(x) e^{- \sqrt{\lambda_1} y}$.     Multiply $ 0 = - \Delta u_e$  by $  \frac{\phi_e}{f(u_e)}$ and integrate this over  the cylinder $ \mathcal{C}$ to obtain
\[ \int_{\Omega } \lambda g(x) \phi = \int_{\mathcal{C}} \frac{\nabla u_e \cdot \nabla \phi_e}{f(u_e)} - \int_{\mathcal{C}} \frac{ | \nabla u_e|^2 \phi_e f'(u_e)}{f(u_e)^2}, \]  and note that the second integral on the right is nonpositive and hence we can rewrite this as

\[ \int_{\Omega } \lambda g(x) \phi  \le \int_{\mathcal{C}} \nabla \phi_e \cdot \nabla h(u_e),\]  where $ h(t) = \int_0^t \frac{1}{f(\tau)} d \tau$.  Integrating the right hand side by parts we have that it is equal to $ \int_{\Omega } \HA \phi h(u)$  which is equal to $ \sqrt{\lambda_1} \int_{\Omega } \phi h(u)$.  So $ h(u) \le C_f$ and hence we have
\[\lambda  \int_{\Omega }  g(x) \phi  \le \sqrt{\lambda_1}  C_f \int_{\Omega } \phi.\]
This shows that $ \lambda^* < \infty$.    The case where $ f$ satisfies (R) and  where $ C_f = \infty$ needs a separate proof, see the proof of (4).   Note that there are  examples of $f$ which satisfy (R) and for which $C_f=\infty$,  for example $ f(t):= (t+1) \log(t+1)+1$.

(3)  The proof in the case where $ g(x)=1$ also works here, see \cite{cdds}.

(4)  Again the proof used in the case where $ g(x)=1$ works to show the monotonicity of $ u_\lambda$, see \cite{cdds},  and hence $ u^*$ is well defined.   One should note that our notion of a weak solution is more restrictive than what is typically used, ie.  we require $ g(x) f(u) \in L^1(\Omega)$ where typically one would only require that $ \delta(x) g(x) f(u)  \in L^1(\Omega)$ where $ \delta(x)$ is the distance from $x$ to $\pOm$.    Hence here our proof will differ from \cite{cdds}.     \\
     Claim:    There exists some $  C < \infty$ such that  
\begin{equation} \label{cla}
\int_\Omega g(x) f'(u_\lambda) f(u_\lambda) \le C, 
\end{equation} for all $ 0 < \lambda < \lambda^*$  (at this point we are allowing for the possibility of $ \lambda^* = \infty$).    We first show that the claim implies that $ \lambda^* < \infty$.   Note that if $\HA \phi = g(x)$ with $ \phi =0 $ on $ \pOm$  then an application of the maximum principle along with the fact that $ f(u_\lambda) \ge 1$ gives  $ u_\lambda \ge \lambda \phi$ in $ \Omega$.   This along with (\ref{cla}) rules out the possibility of $ \lambda^*=\infty$.    Using a proof similar to the one in  \cite{cdds} one sees that $u^*$ is a weak solution to $(P)_{\lambda^*}$ except for the extra integrability condition $ g(x) f(u^*) \in L^1(\Omega)$ that we require.   But sending $ \lambda \nearrow \lambda^*$ in (\ref{cla}) gives us the desired regularity and we are done. \\
 We now prove the claim.   Let $ u=u_\lambda$ denote the minimal solution of $(P)_\lambda$ and let $ u_e$ denote its harmonic  extension.   Take   $ \psi:= f(u_{e})-1$ in (\ref{loc_stable}) ($\psi$ can be shown to be an admissible test function) and write the right hand side as 
\[ \int_{\mathcal{C}} \nabla ( f(u_e)-1) f'(u_e) \cdot  \nabla u_e,\] and integrate this by parts.   Using $(P)_\lambda$ and after some cancellation one arrives at 
\begin{equation} \label{999}
 \int_{\mathcal{C}} (f(u_e)-1) f''(u_e) | \nabla u_e|^2 \le \lambda \int_\Omega g(x) f'(u) f(u).
 \end{equation} 
Define $ H(t):= \int_0^t f''(\tau) ( f(\tau)-1) d \tau $  and so the left hand side of (\ref{999}) can be written as $ \int_{\mathcal{C}} \nabla H(u_e) \cdot \nabla u_e$ and  integrating this by parts gives 
\[ \lambda \int_\Omega g(x) f(u) H(u).\]  Combining this with (\ref{999}) gives 
\begin{equation} \label{1000}
\int_\Omega g(x) f(u) H(u) \le \int_\Omega g(x) f(u) f'(u).
 \end{equation}   To complete the proof we show that $ H(u)$ dominates $ f'(u)$ for big $u$ (resp. $ u$ near $1$) when $f$ satisfies (R)  (resp. (S)).   If $ 0 <T<t$  then one easily sees that 
 \[ H(t) \ge (f(T)-1) ( f'(t) - f'(T)).\]  Using this along with (\ref{1000}) and dividing the domain of $ \Omega$ into regions $ \{ u \ge T\}$ and $ \{u <T \}$ one  obtains the claim.

\hfill $ \Box$

\section{Uniqueness of the extremal solution}\label{section2}

\begin{thm}\label{extremal}  Suppose that either $ f$ satisfies (R) and is log convex or satisfies (S) and is strictly convex.  Then the followings hold.
 \begin{enumerate}
\item[(1)] There are no weak solutions for $(P)_\lambda$ for any $ \lambda > \lambda^*$.

\item[(2)] The extremal solution $u^*$ is the unique weak solution of $(P)_{\lambda^*}$.
\end{enumerate}

\end{thm}

 The following are some properties that the nonlinearity $f$ satisfies. 
\begin{prop}\label{inequality} 
\begin{enumerate}
 \item[(1)] Let $f$ be a log convex  nonlinearity which satisfies (R).
\begin{itemize} 
\item [(i)] For all  $ 0 < \lambda <1$ and $ \delta >0$ there exists $ k>0$ such that
\[  f(\lambda^{-1} t) + k \ge (1+ \delta) f(t) \ \ \ \ \text{for all} \ \ 0 \le t < \infty . \]
\item [(ii)] Given $ \E>0$ there exists  $  0 < \mu <1$ such that
\[ \mu^2 \left( f( \mu^{-1} t) + \E \right) \ge f(t) + \frac{\E}{2} \ \ \ \ \text{for all} \ \ 0 \le t < \infty .\] 
\item [(iii)]   Then $f$ is strictly convex. 
\end{itemize}
\item[(2)] Let $f$ be a nonlinearity which satisfies (S).

 \begin{itemize}   \item [(i)]  Given $ \E>0$ there exists  $  0 < \mu <1$ such that
\[ \mu \left( f( \mu^{-1} t ) + \E \right) \ge f(t) + \frac{\E}{2} \ \ \ \ \text{for all} \ \ 0 \le t \le \mu.\]

\item [(ii)]  Then $ \lim_{t \nearrow 1} \frac{f(t)}{F(t)}= \infty$ where $ F(t):=\int_0^t f(\tau) d \tau$. 
\end{itemize}

\end{enumerate}
\end{prop} 

\begin{proof}

 See \cite{berc}, \cite{luo} for the proof of (1)-(i) and (1)-(ii).  Part (1)-(iii) is trivial.    \\
(2)-(i)  Set $ h(t):= \mu \{ f( \mu^{-1} t) +\E \} - f(t) - \frac{\E}{2}$ and note that $ h'(t) \ge 0$ for all $ 0 \le t \le \mu$ and that $ h(0)>0$ for $ \mu $ sufficiently close to $1$ which gives us the desired result.  \\
(2)-(ii)  Let $ 0 < t <1$ and we use a Riemann sum with right hand endpoints to approximate $ F(t)$.  So for any positive integer $ n$ we have 
\begin{eqnarray*}
 F(t) & \le & 
\frac{t}{n} \sum_{k=1}^n f( \frac{k t}{n}) \le \frac{t(n-1)}{n} f( \frac{(n-1)t}{n}) + \frac{t}{n} f(t),
\end{eqnarray*}

 and so 
\[ \limsup_{t \nearrow 1} \frac{F(t)}{f(t)} \le \frac{1}{n},\] but since $n$ is arbitrary we have the desired result. 

\end{proof}

The following is an essential step in proving Theorem \ref{extremal}. We give the proof of this lemma later.

\begin{lemma}\label{exist} 
Suppose that $f$ is log convex and satisfies (R) or $f$ satisfies (S). Suppose $ \E>0$ and that   $  0 \le \tau  $ is a weak solution of
\begin{eqnarray*}
\hskip 50pt \left\{ \begin{array}{lcl}
\hfill  \HA \tau&=& l(x) \qquad \text{in}\ \Omega \\
\hfill \tau&=& 0 \qquad \quad  \text{on}\ \pOm, 
\end{array}\right.
 \end{eqnarray*} 
  where $   g(x)\left( f(\tau) + \E\right) \le l(x) \in L^1(\Omega)$. Then there exists a regular energy solution of
\begin{eqnarray*}
\hskip 50pt \left\{ \begin{array}{lcl}
\hfill  \HA u&=&  g(x)\left( f(u) + \frac{\E}{2} \right)\qquad \text{in}\ \Omega \\
\hfill u&=& 0 \qquad \qquad \qquad \qquad\text{on}\ \pOm. 
\end{array}\right.
 \end{eqnarray*} 
\end{lemma}

\textbf{Proof of Theorem \ref{extremal}:} Without loss of generality assume that $ \lambda^*=1$ and let $ u^*$ denote the extremal solution of $(P)_{\lambda^*}$.   Suppose that $ v$ is also a weak solution of $(P)_{\lambda^*}$ and $ v$ is not equal to $u^*$. Set $ \Omega_0:=\{ x \in \Omega:  u^*(x) \neq v(x), \;  u^*(x),v(x) \in \mathbb{R} \}$  (resp. $ \Omega_0=\{ x \in \Omega: u^*(x) \neq v(x), \; u^*(x),v(x)<1\}$)  when $f$ satisfies (R)  (resp.  (S))  
 and note that $ | \Omega_0|>0$. Define
\begin{equation*}
h(x):= \left\{
\begin{array}{lc}
\frac{  f(u^*(x))+ f(v(x))}{2} - f( \frac{u^*(x)+v(x)}{2}) & \quad x \in \Omega_0 \\
0 & \qquad  \mbox{ otherwise}.
\end{array}
\right.
\end{equation*}  Note that by the strict  convexity of $ f$, which we obtain either by hypothesis or by Proposition \ref{inequality}, we have $ 0 \le  h $ in $ \Omega$ and $ h >0$ in $ \Omega_0$. Also note that 
$ h \in L^1(\Omega)$.  Define $ z:= \frac{ u^*+v}{2}$. Since $u^*$ and $v$ are weak solutions of  $(P)_{\lambda^*}$, $z$ is a weak solution of

\[ \HA z = g(x) f(z) + g(x) h(x) \qquad \text{in }\  \Omega, \] with $z=0$ on $ \pOm$.  From now on we omit the boundary values since they will always be zero unless otherwise mentioned.  Let $\chi$ and $ \phi$ denote weak solutions of $ \HA \chi = g(x) h(x)$ and $ \HA \phi = g(x) $ in $ \Omega$.  By taking $ \E>0$ small enough one has that $ \chi \ge \E \phi$ in $ \Omega$. Set $ \tau:= z + \E \phi - \chi$ and note that  $ \tau$ is a weak solution of 
\[ \HA \tau = g(x) ( f(z) + \E) \ge 0 \qquad \text{in}\ \Omega,\]
 and by Lemma \ref{max} we have that $ 0 \le \tau$. Moreover, from the fact that $ \tau \le z$ in $\Omega$ we have  $$ g(x) ( f(\tau) + \E) \le \HA \tau \in L^1(\Omega).$$ Applying Lemma \ref{exist}, there exists a regular energy solution $u$ of
\[ \HA u = g(x) ( f(u) + \frac{\E}{2}) \qquad\text{in}\ \Omega.\]  
Set $ w:= u+ \alpha u - \frac{\E}{2} \phi$ where $ \alpha >0$ is chosen small enough such that $ \alpha u \le \frac{\E}{2} \phi$ in $ \Omega$. A straightforward computation shows that $ w$ is a regular energy solution of
\[ \HA w = (1+\alpha) g(x) f(u) +  \frac{\E}{2} \alpha g(x) \qquad\text{in}\ \Omega,\] and  that $ w \le u$ in $ \Omega$.  By  Lemma \ref{max} we also have $ 0 \le w$ in $ \Omega$.   From this we see that $ w$ is a regular energy supersolution of
\[ \HA w \ge (1+\alpha) g(x) f(w) \qquad \text{in}\ \Omega, \] with zero boundary conditions.  We now apply the monotone iteration argument   to obtain a regular energy solution $\tilde{u}$ of $ \HA \tilde{u} = (1+\alpha) g(x) f(\tilde{u}) $ in $ \Omega$  which contradicts the fact that $ \lambda^*=1$.  So, we have shown that $ | \Omega_0|=0$ and so $ u^*=v$ a.e. in $ \Omega$.

\hfill $ \Box$

\textbf{Proof of Lemma \ref{exist}:} Let $ \E >0$ and suppose that $ 0 \le \tau \in L^1(\Omega)$ is a weak solution of $ \HA \tau = l(x) $ in $ \Omega$   where $ 0 \le g(x) ( f( \tau) + \E) \le l(x) $ in $ \Omega$. 
As in the proof of Theorem \ref{extremal}, we omit the boundary values since they will always be Dirichlet boundary conditions and we also assume that $ \lambda^*=1$.  First, assume that $f$ is a log convex nonlinearity which satisfies (R). Let $ u_0:=\tau$ and let $ u_1,u_2, u_3$ be weak solutions of
\[ \HA u_1 = \mu g(x) ( f(u_0) + \E) \qquad \text{in}\ \Omega,\]
\[ \HA u_2 = \mu g(x) ( f(u_1) + \E) \qquad \text{in}\ \Omega,\]
\[ \HA u_3 = \mu g(x) ( f(u_2) + \E) \qquad \text{in}\ \Omega,\]  where $ 0 < \mu <1$ is the constant given in Proposition \ref{inequality} such that $  \mu^2 \left( f( \frac{t}{\mu}) + \E \right) \ge f(t) + \frac{\E}{2} $ for all $ t \ge 0$.   One easily sees that $ u_2 \le u_1 \le \mu u_0$.    Now note that
\begin{eqnarray}\label{sequence}
\nonumber\HA u_1 &=& \mu g(x) ( f(u_0)+\E) \\
& \ge & \mu g(x) \left( f( \frac{u_1}{\mu} ) + \E \right). 
\end{eqnarray}  

By Proposition \ref{inequality} with  $ \delta :=2N-1>0$ and $ 0<\lambda = \mu<1$  there exists some $ k>0$ such that 
\[ f( \frac{u_1}{\mu})  \ge 2N f(u_1)-k,\] 
hence one can rewrite (\ref{sequence}) as 
\begin{eqnarray*}
\HA u_1 & \ge & \mu g(x) \left( 2Nf( u_1 ) -k+ \E \right). 
\end{eqnarray*}

We let $ \phi$ be as in the proof of Theorem \ref{extremal} and examine $ u_1 + t \phi$ where $ t>0$ is to be picked later.  Note that
\begin{eqnarray*}
\HA (u_1 + t \phi)  &=& \HA u_1 + t g(x)  \\
&\ge& 2N \mu g(x) \left( f( u_1 ) + \E \right) +m  g(x),
\end{eqnarray*}  

 where $ m:= t - \mu k  + \E \mu (1-2N)$ and we now pick $ t>0$ big enough   such that $ m=0$.  Therefore, from the definition of $u_2$ we have that
\[ \HA (u_1 + t \phi) \ge 2N\  \HA u_2 \qquad\text{in}\  \Omega. \]  
So, from the maximum principle we get $$ u_2 \le \frac{1}{2N}(u_1 + t\phi) \qquad\text{in}\  \Omega.$$
    Since $f$ is log convex, there is some smooth, convex increasing function $ \beta$ with $ \beta(0)=0$ and $ f(t) = e^{\beta(t)}$.   By the convexity of $ \beta$ and since $ \beta(0)=0$, we have
\[ \beta(u_2) \le \frac{ 1}{2N} \beta(u_1 + t \phi) \le \frac{ 1}{2N} \beta(\mu u_0 + t \phi),\]
 but
\[ \beta(  \mu u_0 + t \phi) = \beta( \mu u_0 + (1-\mu)  \frac{t \phi}{1-\mu} ) \le \mu \beta(u_0) + (1-\mu) \beta( \frac{t \phi}{1-\mu}).\] 
From this we can conclude
\[ f(u_2)^{2N} \le e^{\mu \beta(u_0)} e^{(1-\mu) \beta( \frac{t \phi}{1-\mu})} \le f(u_0) f( \frac{t\phi}{1-\mu})^{1-\mu}.\]

So, we see that $ g(x) f(u_2)^{2N} \le C  g(x) f(u_0) \in L^1(\Omega)$  for some large constant $C$.

 Since $ g(x)$ is bounded,  we conclude that $ g(x) f(u_2) \in L^{2N}(\Omega)$.   But $ u_3$ satisfies  
  $ \HA u_3 = \mu g(x) ( f(u_2) +\E) $ in $ \Omega$ and so by elliptic regularity we have that $ u_3$ is bounded (since the right hand side is an element of $ L^p(\Omega)$ for some $ p>N$)  
    and  now we use the fact that $ 0 \le u_3 \le u_2$ and the monotone iteration argument to obtain a regular energy solution $w$ to $ \HA w = \mu g(x) ( f(w) +\E)$ in $ \Omega$.  

Now, set $ \xi:= \mu w$ and note that  $ \xi$ is a regular energy solution of
\[ \HA \xi= \mu^2 g(x) \left(  f(\frac{\xi}{\mu}) + \E\right) \qquad \text{in}\ \Omega, \] and  from Proposition \ref{inequality}, we have
\[ \HA \xi \ge g(x) \left( f(\xi) + \frac{\E}{2}\right) \qquad \text{in}\ \Omega, \] and so by an iteration argument, we have the desired result.

Now, assume that $f$ satisfies (S).  In this case, the proof is much simpler. Define $w:=\mu \tau$ where $ 0 < \mu <1$ is from  Proposition \ref{inequality}.   Then note that $ 0 \le w \le \mu$ a.e. and 
 
\begin{eqnarray*}
\HA w =\mu  l(x) &\ge& \mu g(x) ( f(\frac{w}{\mu}) + \E)\\
&\ge& g(x) ( f({w}) +\frac{\E}{2}) .
\end{eqnarray*} 
Hence, $w$ is a regular  energy supersolution of 
\begin{eqnarray*}
\HA w&\ge& g(x) ( f({w}) +\frac{\E}{2}),
\end{eqnarray*} 
 and we have the desired result after an application of the monotone iteration argument. 

\hfill $ \Box$

\section{Uniqueness of solutions for small $ \lambda$}\label{section3}

 In this section we prove uniqueness theorems for equation $(P)_\lambda$ for small enough $\lambda$.  Throughout this section we assume that $ g=0 $ on $ \pOm$.   We need the following regularity result.  
 
 \begin{prop} \label{reg} \cite{ct} Let $\alpha\in (0,1)$, $\Omega$ be a $C^{2,\alpha}$ bounded domain in $\mathbb{R}^N$ and suppose that $u$ is a weak solution of $ \HA u = h(x)$ in $ \Omega$ with $ u=0 $ on $ \pOm$. 
 \begin{enumerate} 
 \item[(1)]  Suppose that $ h \in L^\infty(\Omega)$.  Then $ u_e \in C^{0,\alpha}(\overline{\mathcal{C}})$ hence $ u \in C^{0,\alpha}(\overline{\Omega})$. 
 \item[(2)] Suppose that $ h \in C^{k,\alpha}(\overline{\Omega}) $ where $k=0$ or $k=1$  and $ h=0 $ on $ \pOm$.  Then $ u_e \in C^{k+1,\alpha}(\overline{\mathcal{C}})$ hence $u \in C^{k+1,\alpha}(\overline{\Omega})$   . 
 \end{enumerate} 
 \end{prop}  
 
 Using this one easily obtains the following: 
 \begin{cor} \label{cor}   For each $ 0 < \lambda < \lambda^*$ the minimal solution of $(P)_\lambda$, $u_\lambda$,  belongs to $ C^{2,\alpha}(\overline{\Omega})$.   In addition $ u_\lambda \rightarrow 0$ in $ C^1(\overline{\Omega})$ as $ \lambda \rightarrow 0$.
 \end{cor}
 
We now come to our main theorem of this section.

\begin{thm} \label{small}   Suppose that  $ \Omega$ is a  star-shaped domain  with respect to the origin and set $ \gamma:= \sup_\Omega \frac{ x \cdot \nabla g(x)}{g(x)}$.  
\begin{enumerate}  
\item [(1)]Suppose that $f$ satisfies (R) and that 
\begin{equation} \label{sh}
\limsup_{t \rightarrow\infty}\frac{F(t)}{ f(t) t} <\frac{N-1}{2(N+\gamma)}.
\end{equation}  Then for sufficiently small $ \lambda$, $ u_\lambda$ is the unique regular energy  solution of $(P)_\lambda$.   

\item [(2)] Suppose that $ f$ satisfies (S).  Then  for sufficiently small $ \lambda$, $ u_\lambda$ is the unique   regular energy  solution $(P)_\lambda$.   

\end{enumerate} 

\end{thm}

\textbf{Proof:}  Let $f$ satisfy (R) and (\ref{sh}) or let $f$ satisfy (S)  and suppose that $u$ is a second regular energy solution of $(P)_\lambda$ which is different from the minimal solution $ u_\lambda$.   
 Set $ v: =u-u_\lambda$ and note that $ v \ge 0$ by the minimality of $u_\lambda$  and $ v \neq 0$ since $u$ is different from the minimal solution.

   A computation shows that  $v$ satisfies the equation
 \begin{equation}\label{minuseq}
 (-\Delta)^\frac{1}{2} v = \lambda g(x) \left\{ f(u_\lambda +v) - f(u_\lambda) \right\}.
 \end{equation}   Applying  Proposition \ref{reg} to $u$ and $u_\lambda$ separately shows that $ v_e \in C^{2,\alpha}(\overline{ \mathcal{C}})$.

A computation shows  the following identity holds 
$$  \mbox{div} \{  (z,\nabla v_e) \nabla v_e -z \frac{|\nabla v_e|^2}{2}\}  +\frac{N-1}{2} |\nabla v_e|^2 = (z,\nabla v_e) \Delta v_e,$$
where $z=(x,y)$. Integrating this identity over $\Omega \times (0,R)$  we end up with 
\begin{equation} \label{tan}
\frac{1}{2} \int_{\pOm \times (0,R)} | \nabla v_e|^2\ x \cdot \nu + \int_{\Omega } x \cdot \nabla_x v_e \ \partial_\nu v_e + \frac{N-1}{2} \int_{\Omega \times (0,R)} | \nabla v_e|^2 + \E(R) =0,
\end{equation} where
\[ \E(R):= \int_{ \Omega \times \{y=R\}} \left( x \cdot \nabla_x v_e + R\ \partial_y v_e\right) \partial_y v_e -\frac{R}{2} | \nabla v_e|^2.\]
One can show that $ \E(R) \rightarrow 0$ as $ R \rightarrow \infty$,  for details on this and  the above calculations see \cite{TAN}.  Sending $ R \rightarrow \infty$ and since  $ \Omega$ is star-shaped with respect to the origin, we have 
\[ \frac{N-1}{2} \int_{\mathcal{C}} | \nabla v_e|^2 \le - \int_{\Omega } x \cdot \nabla_x v \ \partial_\nu v_e,\] 
and after using (\ref{minuseq}) one obtains

 \begin{eqnarray} \label{return}
 \frac{N-1 }{2}\int_{\mathcal{C}} | \nabla v_e|^2   \le  -  \lambda \int_{\Omega  } x \cdot \nabla_x v \   g(x) \{ f(u_\lambda + v) - f(u_\lambda)\}.
 \end{eqnarray}

 We now compute the right hand side of (\ref{return}).  Set $ h(x,\tau):= f(u_\lambda(x) + \tau) - f(u_\lambda(x)) $ and let $ H(x,t) = \int_0^t h(x,\tau) d \tau$.   For this portion of the proof we are working on $ \Omega$ and hence all gradients are with respect to the $x$ variable.   To clarify our notation note that the chain rule can be written as 
 \[  \nabla H(x,v) = \nabla_x H(x,v) + h(x,v) \nabla v,\] where we recall $ v=v(x)$.  
 Some computations now show that
\[ H(x,t)= F(u_\lambda +t) - F(u_\lambda) -f(u_\lambda) t, \]  and
\[ \nabla_x H(x,t) = \{ f(u_\lambda +t) - f(u_\lambda) -f'(u_\lambda) t \} \nabla u_\lambda,\]

and so the right hand side of (\ref{return}) can be written as
\begin{eqnarray*}
- \lambda \int_{\Omega  }   g(x) \{ f(u_\lambda + v) - f(u_\lambda) \}  x \cdot \nabla v &=& - \lambda \int_{\Omega }  g(x) h(x,v) x \cdot\nabla v \\
&=& -\lambda \int_{\Omega }   g(x) x \cdot \{ \nabla H(x,v) - \nabla_x H(x,v) \} \\
&=& \lambda \int_{\Omega }  g(x) x \cdot \nabla_x H(x,v) +\lambda N \int H(x,v) g(x) \\&&+ \lambda \int_{\Omega }  H(x,v) x\cdot \nabla g(x).
\end{eqnarray*}

Therefore, (\ref{return}) can be written as
\begin{eqnarray} \label{poop}
\frac{N-1 }{2}\int_{\mathcal{C}} | \nabla v_e|^2  & \le & \lambda \int_{\Omega }  x \cdot \nabla u_\lambda g(x) \{ f(u_\lambda +v) - f(u_\lambda) - f'(u_\lambda)v\} \nonumber 
   \\ &&+  N \lambda\int_{\Omega }  g(x) \{ F(u_\lambda+v) - F(u_\lambda) - f(u_\lambda)v\} \nonumber \\ &&+ \lambda \int_{\Omega }   x \cdot \nabla g(x) \{ F(u_\lambda+v) - F(u_\lambda) - f(u_\lambda)v \}. 
\end{eqnarray}   We now assume we are in case (1).    
 Let $ \alpha$ be such that 
\[ \limsup_{\tau \rightarrow \infty} \frac{F(\tau)}{ \tau f(\tau)} < \alpha < \frac{N-1}{2(N+\gamma)},\] so there exists some $ \tau_0 >0$ such that $ F(\tau) < \alpha \tau f(\tau) $ for all $ \tau \ge \tau_0$.  Let $ 0 < \theta <1$ be such that $ \frac{\theta (N-1)}{2} - \alpha (N+\gamma)>0$ and we now decompose the left hand side of (\ref{poop}) into the convex combination

\begin{equation} \label{pooooP}
 \frac{\theta (N-1)}{2}  \int_{\mathcal{C}} | \nabla v_e|^2  + \frac{(N-1)(1-\theta)}{2} \int_{\mathcal{C}} | \nabla v_e|^2.  
 \end{equation}

Using the following trace theorem: there exists some $ \tilde{C}>0$ such that  
\begin{equation} \label{trace}
\int_{\mathcal{C}} | \nabla w|^2 \ge \tilde{C} \int_\Omega w^2, \qquad \forall w \in H^1_{0,L}(\mathcal{C}),
\end{equation} 
one sees that (\ref{pooooP})
  is bounded below by  
\[ \frac{\theta (N-1)}{2}  \int_{\mathcal{C}} | \nabla v_e|^2  +  C \int_{\Omega  } v^2.\]    By taking $ C>0$ smaller if necessary one can bound this from below by 
 \[  \frac{\theta (N-1)}{2}  \int_{\mathcal{C}} | \nabla v_e|^2  +  C \int_{\Omega  } g(x) v^2, \] and after using   (\ref{minuseq}), this last quantity is equal to 
 \begin{equation} \label{zs} 
 \frac{ \lambda \theta (N-1)}{2} \int_{\Omega } g(x) \{ f(u_\lambda +v) - f(u_\lambda)\} v + C \int_{\Omega  } g(x) v^2.
 \end{equation}  Substituting (\ref{zs}) into (\ref{return}) and rearranging one arrives at an inequality of the form  
 \[ \int_{\Omega  } g(x) T_\lambda(x,v) \le 0,\]   where 
 \begin{eqnarray*}
 T_\lambda(x,\tau) &=& \frac{\theta (N-1)}{2} \{ f(u_\lambda + \tau) - f(u_\lambda)\} \tau + \frac{C}{\lambda}\tau^2 \\
 && - N \{ F(u_\lambda+\tau) - F(u_\lambda) - f(u_\lambda) \tau\} \\
 &&- \frac{x \cdot \nabla g}{g} \{ F(u_\lambda +\tau)- F(u_\lambda) - f(u_\lambda) \tau\} \\
 && -  x \cdot \nabla u_\lambda  \{ f(u_\lambda+ \tau) - f(u_\lambda) -f'(u_\lambda) \tau\}.
 \end{eqnarray*}   To obtain a contradiction we show that for sufficiently small $ \lambda>0$ that $ T_\lambda(x,\tau) >0$ on $(x,\tau) \in \Omega \times (0,\infty)$ and hence we must have that $ v=0$.    Define 
 \begin{eqnarray*}
 S_\lambda(x,\tau) &=& \frac{\theta (N-1)}{2} \{ f(u_\lambda + \tau) - f(u_\lambda)\} \tau + \frac{C}{\lambda}\tau^2 \\
 && - (N+\gamma) \{ F(u_\lambda+\tau) - F(u_\lambda) - f(u_\lambda) \tau\} \\
 && - \E_\lambda \{ f(u_\lambda+ \tau) - f(u_\lambda) -f'(u_\lambda) \tau\}.
 \end{eqnarray*}  where $ \E_\lambda:= \| \nabla u_\lambda \cdot x\|_{L^\infty}$.  Note that since $ f$ is increasing and convex that $ T_\lambda(x,\tau) \ge S_\lambda(x,\tau)$ for all $ \tau \ge 0$.   We now show the desired positivity for $ S_\lambda$ and to do this we examine large and small $ \tau$ separately.  \\
  
 Large $\tau$:  Take $ \tau \ge \tau_0$ and $ 0 < \lambda \le \frac{\lambda^*}{2}$.   Since $ f$ is convex and increasing 
 \begin{eqnarray} \label{cra}
 S_\lambda(x,\tau)& \ge & \frac{\theta (N-1)}{2} f(u_\lambda+\tau) \tau -(N+\gamma) F(u_\lambda+\tau)  \nonumber \\
 && - \E_\lambda f(u_\lambda+\tau) + \frac{C}{\lambda} \tau^2 \nonumber  \\
 &&- \frac{\theta (N-1)}{2} f(u_\lambda) \tau, 
 \end{eqnarray}  
  but $ F(u_\lambda + \tau) < \alpha (u_\lambda+ \tau) f(u_\lambda + \tau)$ for all $ \tau \ge \tau_0$ and so the right hand side of (\ref{cra}) is bounded below by 
 \begin{eqnarray*}
 && f(u_\lambda+\tau) \left[ \tau \left\{ \frac{\theta(N-1)}{2} - (N+\gamma)\alpha \right\} -\E_\lambda - (N+\gamma) \alpha u_\lambda \right]  \\
 && - \frac{\theta (N-1)}{2} f(u_\lambda) \tau + \frac{C}{\lambda} \tau^2.
 \end{eqnarray*}  Using the fact that $f$ is superlinear at $ \infty$ there exists some $ \tau_1 \ge \tau_0$ such that $ S_\lambda(x,\tau) >0$ for all $ \tau \ge \tau_1$ and $ 0 < \lambda \le \frac{\lambda^*}{2}$.    \\
 Small $ \tau$:    Let $ 0 < \lambda_0 < \frac{\lambda^*}{2}$ be such that $ \| u_\lambda\|_{L^\infty} \le 1$.  Using the convexity and monotonicity of $f$ and  Taylor's Theorem there exists some  $ C_1>0$ such that 
 \[   F(u_\lambda+\tau)-F(u_\lambda) - f(u_\lambda) \tau \le C_1 \tau^2, \qquad  f(u_\lambda+\tau) - f(u_\lambda) - f'(u_\lambda) \tau \le C_1 \tau^2,\] 
  for all $ 0 \le \tau \le \tau_0$,  $ 0 < \lambda \le \lambda_0$ and $ x \in \Omega$.   Noting that the first term of $S_\lambda(x,\tau)$ is positive for $ \tau >0$ one  sees that for all $ 0 < \tau  \le \tau_0$, $ x \in \Omega$ and $ 0 < \lambda < \lambda_0$ one has the lower bound 
  \[ S_\lambda(x,\tau) \ge \frac{C}{\lambda} \tau^2 -(N+\gamma+\E_\lambda) C_1 \tau^2,\]  and hence by taking $ \lambda $ smaller if necessary we have the desired result.  \\
  (2)  We now assume that $f$ satisfies (S).  One uses a similar approach to arrive at an inequality of the form 
  \[ \int_\Omega T_\lambda(x,v) \le 0,\] where as before $ v=u-u_\lambda \ge 0$ and where we assume that $ v \neq 0$.   
To arrive at a contradiction we  show that for sufficiently small $ \lambda$ that  $ T_\lambda(x,\tau)>0$ for all $ x \in \Omega$ and for all $ 0 < \tau < 1 -u_\lambda(x)$.   Again the idea is to break the interval into 2 regions.   For $ \tau$ such that $ \tau + u_\lambda(x)$ close to $1$ we use Proposition \ref{inequality},  2 (ii) to see the desired positivity.  For the remainder of the interval we again use Taylor's Theorem.

\hfill $ \Box$

 \end{document}